\documentclass[12pt]{article}

\usepackage[T1]{fontenc}
\usepackage{lmodern}
\usepackage[margin=1in]{geometry}
\usepackage{microtype}
\usepackage{amsmath,amssymb,amsthm,mathtools}
\usepackage{booktabs}
\usepackage{array}
\usepackage{enumitem}
\usepackage{xcolor}
\usepackage[round,authoryear]{natbib}
\usepackage{tikz}
\usepackage{pgfplots}
\pgfplotsset{compat=1.18}
\usetikzlibrary{arrows.meta,positioning,calc,fit,decorations.pathreplacing,patterns,backgrounds}
\usepackage{listings}
\usepackage[hidelinks]{hyperref}

\newtheorem{theorem}{Theorem}
\newtheorem{lemma}[theorem]{Lemma}

\theoremstyle{remark}

\newcommand{\FDR}{\operatorname{FDR}}
\newcommand{\FDP}{\operatorname{FDP}}

\newcommand{\1}{\mathbf 1}
\newcommand{\R}{\mathbb R}
\newcommand{\Pp}{\mathbb P}
\newcommand{\E}{\mathbb E}
\newcommand{\Phibar}{\overline\Phi}
\newcommand{\eps}{\varepsilon}
\newcommand{\doi}[1]{\href{https://doi.org/#1}{doi:\nolinkurl{#1}}}
\newcommand{\arxiv}[1]{\href{https://arxiv.org/abs/#1}{arXiv:\nolinkurl{#1}}}

\definecolor{bhblue}{RGB}{43,108,176}
\definecolor{bhorange}{RGB}{221,126,32}
\definecolor{bhgreen}{RGB}{40,145,93}
\definecolor{bhred}{RGB}{190,66,58}
\definecolor{bhpurple}{RGB}{117,91,164}
\definecolor{bhgray}{RGB}{95,105,115}

\tikzset{
  bhflow/.style={
    draw=bhgray,
    rounded corners=2pt,
    thick,
    align=center,
    minimum height=1.45cm,
    text width=2.55cm,
    fill=bhgray!5,
    font=\small
  },
  bhblock/.style={
    draw=bhgray,
    rounded corners=2pt,
    thick,
    align=left,
    text width=7.0cm,
    minimum height=1.35cm,
    inner sep=6pt,
    font=\small
  },
  bhnoise/.style={
    draw=bhgray,
    rounded corners=2pt,
    align=center,
    text width=2.35cm,
    minimum height=1.05cm,
    fill=bhgray!4,
    font=\scriptsize
  }
}

\lstdefinestyle{certificate}{
  language=Python,
  basicstyle=\ttfamily\scriptsize,
  keywordstyle=\bfseries,
  commentstyle=\itshape,
  columns=fullflexible,
  keepspaces=true,
  breaklines=true,
  breakatwhitespace=false,
  showstringspaces=false,
  frame=single,
  xleftmargin=0.5em,
  xrightmargin=0.5em,
  aboveskip=1em,
  belowskip=1em
}

\title{The Benjamini--Hochberg Procedure Can Fail to Control the FDR
for Correlated Two-Sided Gaussian Tests}
\author{
Edgar Dobriban\footnote{
Department of Statistics and Data Science, University of Pennsylvania. E-mail address: \texttt{dobriban@wharton.upenn.edu}.}
}
\date{\today}

\begin{document}
\maketitle

\begin{abstract}
We show that the Benjamini--Hochberg procedure can fail to control
the false discovery rate (FDR) at its nominal level for correlated
two-sided Gaussian $p$-values.  We construct a factor model for which,
at level $\alpha=0.01$, a rigorous interval-arithmetic certificate proves
$\FDR>0.0104$ for all sufficiently large numbers of hypotheses.
This disproves a conjecture widely believed to be true for twenty years.
Monte Carlo experiments are consistent with the theoretical result.
The proof was obtained by GPT-5.6 Pro and carefully checked by the author.
\end{abstract}

%\tableofcontents

\section{The conjecture and the counterexample}

Multiple hypothesis testing is a central topic in modern statistics, with
applications in biology, genomics, astronomy, economics, finance, and many
other fields.  Over the last three decades, controlling the \emph{false
discovery rate} (FDR) \citep{BenjaminiHochberg1995} has become a standard
target in applications involving many tests.

The Benjamini--Hochberg (BH) procedure is the standard method for FDR control.
It controls the FDR when the $p$-values are independent
\citep{BenjaminiHochberg1995}, or when they satisfy the weaker positive
regression dependence condition of \citet{BenjaminiYekutieli2001}.  See also
\citet{Sarkar2002,BlanchardRoquain2008,FinnerDickhausRoters2007} and the
references below.

A crucial setting that had remained unresolved is that of correlated two-sided
Gaussian tests.  This setting matters because many data sets involve correlated
tests; for example, neighboring genes or genetic variants can be correlated.
Two-sided tests are also common because the direction of an effect is often
unknown in advance, requiring simultaneous sensitivity to positive and
negative effects.

In this paper we show, contrary to prior conjectures and supporting evidence
\citep[e.g.,][]{ReinerBenaim2007,Benjamini2010}, that the Benjamini--Hochberg
procedure need not control the FDR at its nominal level for correlated
two-sided Gaussian tests.

More specifically, we consider the following setting. 
Let $X\sim \mathcal{N}_m(\mu,\Sigma)$, where $\Sigma$ is a correlation matrix.  For
$1\le i\le m$, form the two-sided Gaussian $p$-value
$
 P_i=2\Phibar(|X_i|),
$
where $\Phi$ is the standard normal CDF and $\Phibar=1-\Phi$.  
The Benjamini--Hochberg (BH) procedure \citep{BenjaminiHochberg1995}
at
level $\alpha \in (0,1)$
uses critical values $t_r=\alpha r/m$.  
Let
$
 R=\max\bigl\{r:\#\{i:P_i\le t_r\}\ge r\bigr\},
$
with $R=0$ if the set is empty.
The BH procedure rejects  those hypotheses 
$H_i:\mu_i=0$
with
$P_i\le t_R$.  
If $I_0=\{i:\mu_i=0\}$ denotes the true-null index set, define the
number of false rejections by
$
 V=\#\{i\in I_0:P_i\le t_R\}.
$
The false discovery proportion (FDP), the fraction of rejected hypotheses that
are true nulls, is $\FDP=V/\max(R,1)$, and the false discovery rate is its mean,
$\FDR=\E[\FDP]$.

A key question about the BH procedure is the following:

\begin{center}
        \emph{Does the Benjamini--Hochberg procedure control the false discovery rate at level $\alpha$, i.e., do we have $\FDR\le\alpha$, 
        for every $m$, every mean vector, every correlation matrix,
and every $\alpha\in(0,1)$?}
\end{center}

Before this work, a positive answer was widely believed.
\citet{farcomeni2006more,ReinerBenaim2007} provided empirical and theoretical
evidence in support of FDR control, and \citet{kim2008effects} provided
additional evidence through extensive simulations.  In an influential review,
\citet{Benjamini2010} wrote that ``\emph{convincing simutheoretical evidence
indicates that [FDR control] holds for two-sided z-tests with any correlation
structure}.''  This motivates referring to the positive answer as a conjecture.

More recently, \citet{Sarkar2023} wrote that ``\emph{The answer to this question
is generally believed to be yes, and is conjectured so in the literature since
results of numerical studies investigating the question and reported in
numerous papers strongly support it}.''  \citet{Sarkar2023} continued that
``\emph{proving this conjecture [\ldots] seems an urgent and important undertaking.}''
Similar statements were made by \citet{SarkarZhang2025} and
\citet{GhoshSarkar2025}.

%Beyond these statements,  

In this paper we show that the conjecture fails, by constructing the following Gaussian factor model. 
 Let
$
 Z,\,(\eps_i)_{i\ge1},\,(\eta_j)_{j\ge1},\,(\xi_k)_{k\ge1}
$
be mutually independent standard normal random variables.  For a fixed $N$,
define three coordinate blocks by
\begin{align}
  X_i^{(0)}
  &=\frac{3}{10}Z+\frac{\sqrt{91}}{10}\eps_i,
  &&1\le i\le96N,\label{eq:block0}\\
  X_j^{(1)}
  &=\frac{12}{5}-\frac{3}{10}Z+\frac{\sqrt{91}}{10}\eta_j,
  &&1\le j\le N,\label{eq:block1}\\
  X_k^{(2)}
  &=\frac{22}{5}-\frac{18}{25}Z+\frac{\sqrt{301}}{25}\xi_k,
  &&1\le k\le3N.\label{eq:block2}
\end{align}
The block in \eqref{eq:block0} contains the $96N$ true nulls.  The other two
blocks contain the $4N$ nonnulls.  
% Figure~\ref{fig:factor-model} displays the
% sign pattern that drives the construction:
The null block moves with $Z$,
whereas both signal blocks move against it, at two different strengths.

\begin{theorem}[The Benjamini--Hochberg Procedure Can Fail to Control the FDR
for Correlated Two-Sided Gaussian Tests]\label{thm:main}
For each integer $N\ge1$, let $m_N=100N$, let the first $96N$ hypotheses be
true nulls, and let $\FDR_N$ denote the FDR in the model above.  Then the
Benjamini--Hochberg procedure at level $\alpha=0.01$ satisfies, for all
sufficiently large $N$,
$$
  \FDR_N>0.0104>\alpha.
$$
Consequently, the conjecture is false.
\end{theorem}

\begin{proof}
The construction and all analytic details occupy the remainder of the paper.
The final strict numerical inequality is established by the complete
outward-rounded certificate in Appendix~\ref{app:certificate}.
\end{proof}

\paragraph{AI usage.}
The proof was obtained by GPT-5.6 Pro.  The model was asked directly to prove
or disprove the conjecture and was provided only with the mathematical
definition of the Benjamini--Hochberg procedure.  After about 90 minutes of
reasoning, the model produced a proof, an example, and code for the numerical
certificate, which form the basis of this paper.\footnote{The conversation is
available as a \href{https://chatgpt.com/share/6a541c6f-a2d0-83ea-bb2f-782271a103ca}{shared ChatGPT conversation}.}
The author carefully checked the entire argument and the associated numerical
certificate.  Subsequently, the author asked the model to provide additional
simulations, related work, and illustrations for a paper draft, and wrote the
final version by editing the AI-generated draft.

Thus, this work falls into a line of work 
where AI models have helped professional mathematical scientists
resolve open problems, see e.g., \cite{feldman2025g,jang2025point,salim2025accelerating,bubeck2025early,alexeev2025asymptotically,alexeev2025forbidden,dobriban2025solving,abouzaid2026first,openai2026unitdistance,wang2026adaboost,openai2026cycle}, etc.

\section{Relation to existing FDR analysis}
\label{sec:innovation}

The conjecture lies between two classical regimes.  Under mutual independence
of the null $p$-values and independence from the nonnull $p$-values, ordinary
BH satisfies the sharper bound $\FDR\le\pi_0\alpha$, where
$\pi_0=|I_0|/m$ \citep{BenjaminiHochberg1995}.  The same bound holds under
positive regression dependence on the subset of true nulls (PRDS)
\citep{BenjaminiYekutieli2001,Sarkar2002,BlanchardRoquain2008}. 

By contrast,
under completely arbitrary dependence the universal finite-sample guarantee
for the unmodified linear step-up rule carries a harmonic inflation factor;
replacing $\alpha$ by $\alpha/H_m$, where $H_m=\sum_{r=1}^m r^{-1}$, restores
level-$\alpha$ control \citep{BenjaminiYekutieli2001}.  

For one-sided Gaussian tests, nonnegative correlations provide an important
PRDS setting \citep{BenjaminiYekutieli2001,Sarkar2002}.  The two-sided
transformation is qualitatively different:
$\{P_i\le t\}=\{X_i\ge c_t\}\cup\{X_i\le-c_t\}$ folds together two tails that
induce opposite conditional shifts in correlated coordinates.  Consequently,
the usual monotone-regression and total-positivity arguments used for one-sided
statistics do not automatically transfer to the folded Gaussian vector. 
This
obstruction is discussed explicitly by \citet{SarkarZhang2025} and
\citet{GhoshSarkar2025}.  

Earlier work of \citet{ReinerBenaim2007} combined
low-dimensional analysis, upper bounds, and simulation evidence, and
\citet{Benjamini2010} described the evidence for arbitrary-correlation control
as ``simutheoretical'' while noting that a complete proof was unavailable.
The later literature therefore developed dependence-adjusted or shifted
alternatives with provable control rather than a proof for ordinary BH
\citep{FithianLei2022,Sarkar2023,SarkarZhang2025,GhoshSarkar2025}. 

Our argument uses a classical
empirical CDF crossing representation of BH.  In
independent mixture models, and in several dependent asymptotic regimes, the
random BH threshold is compared with the crossing of a limiting $p$-value CDF
and the line $t/\alpha$
\citep{GenoveseWasserman2002,GenoveseWasserman2004,StoreyTaylorSiegmund2004,FinnerDickhausRoters2007}.

An innovation in our setting is the construction of a specific Gaussian factor model. 
 Conditioning on a single latent factor produces three
independent within-block empirical processes, but leaves a \emph{random}
limiting CDF $G_Z$.  The loadings and nonnull means are tuned so that, on a
set of latent-factor values of positive Gaussian probability, the nonnull
blocks enlarge the BH rejection count at precisely the same time that the
conditional null distribution has heavier two-sided tails.  This creates a
self-consistent threshold with an FDP slightly above $\alpha$.

A second innovation is that the argument 
avoids requiring 
a unique limiting
crossing, differentiability at a crossing, or convergence of the BH threshold
to an explicitly solved root.  Two strict sign conditions merely bracket the
threshold.  Monotonicity of Gaussian tail probabilities then turns the
continuum of $(z,c)$ values into finitely many rational rectangles, and
outward-rounded ball arithmetic verifies the required signs.   This combination of a one-sided threshold bracketing and a finite
interval certificate is the distinctive mechanism of the proof.

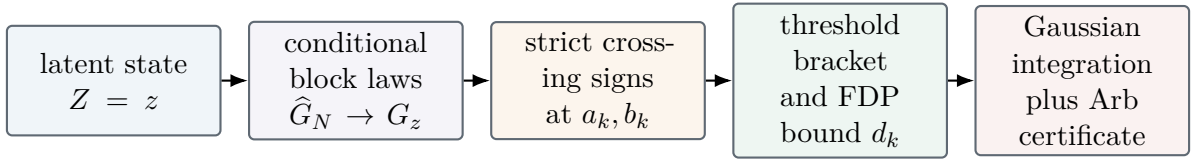
\begin{figure}[htbp]
\centering
\begin{tikzpicture}[node distance=0.35cm and 0.35cm, >=Latex]
  \node[bhflow,fill=bhblue!7] (latent)
    {latent state\\$Z=z$};
  \node[bhflow,right=of latent,fill=bhpurple!7] (limits)
    {conditional block laws\\$\widehat G_N\to G_z$};
  \node[bhflow,right=of limits,fill=bhorange!8] (signs)
    {strict crossing signs\\at $a_k,b_k$};
  \node[bhflow,right=of signs,fill=bhgreen!8] (fdp)
    {threshold bracket\\and FDP bound $d_k$};
  \node[bhflow,right=of fdp,fill=bhred!7] (integrate)
    {Gaussian integration\\plus Arb certificate};
  \draw[->,thick] (latent) -- (limits);
  \draw[->,thick] (limits) -- (signs);
  \draw[->,thick] (signs) -- (fdp);
  \draw[->,thick] (fdp) -- (integrate);
\end{tikzpicture}
\caption{Proof architecture.  The latent factor is not averaged out at the
start; it indexes a family of deterministic limiting BH problems whose
pointwise lower bounds are integrated only at the end.}
\label{fig:proof-architecture}
\end{figure}

\section{Lemmas for the asymptotic BH threshold}

In the remaining sections we present the argument of the proof. 
The proof is transparent in terms of empirical $p$-value distribution
functions.  Interpreting BH as the rightmost crossing of an empirical CDF and
the line $t/\alpha$ is standard in the asymptotic FDR literature
\citep{GenoveseWasserman2002,GenoveseWasserman2004,StoreyTaylorSiegmund2004}.
The following lemma isolates the
weaker one-sided bracketing
fact that will be needed here.

For a sample of size $M_N$, let
$$
  \widehat G_N(t)=\frac1{M_N}\sum_{i=1}^{M_N}\1\{P_{i,N}\le t\},
  \qquad 0\le t\le1.
$$
Let the BH grid be
$
  \mathcal T_N=\left\{\frac{\alpha r}{M_N}:1\le r\le M_N\right\},
$
and define the BH $p$-value threshold
$$
  \tau_N=\max\left(
  \left\{t\in\mathcal T_N:\widehat G_N(t)\ge\frac{t}{\alpha}\right\}
  \cup\{0\}\right).
$$
This is exactly $\tau_N=\alpha R_N/M_N$: at the grid point
$t=\alpha r/M_N$, the inequality $\widehat G_N(t)\ge t/\alpha$ is equivalent
to having at least $r$ $p$-values below the $r$th BH critical value.
 The argument will leverage the two lemmas below, 
whose proofs
 are presented in the appendix. 

\begin{lemma}[Two strict sign conditions bracket the BH threshold]
\label{lem:threshold}
Fix $\alpha\in(0,1)$.  Suppose $M_N\to\infty$ and, for a continuous
distribution function $G$ on $[0,1]$,
$
  \|\widehat G_N-G\|_\infty\longrightarrow0.
$
Let $0<v\le w\le\alpha$.  If
\begin{equation}
  G(t)<\frac{t}{\alpha}
  \quad\text{for every }t\in[w,\alpha],
  \label{eq:no-feasible}
\end{equation}
and
$G(v)>\frac{v}{\alpha}$,
then $\limsup_{N\to\infty}\tau_N\le w$ and
$\liminf_{N\to\infty}\tau_N\ge v$.
\end{lemma}

We also need the corresponding lower bound for the false discovery
proportion.  Suppose that $|I_{0,N}|=\pi_0M_N$ for every $N$, where
$\pi_0\in(0,1]$, and define the true-null empirical CDF
$$
  \widehat F_{0,N}(t)
  =\frac1{\pi_0M_N}\sum_{i\in I_{0,N}}\1\{P_{i,N}\le t\}.
$$

\begin{lemma}[Threshold bracketing implies an FDP lower bound]
\label{lem:fdp}
In addition to the assumptions of Lemma~\ref{lem:threshold}, suppose
$
  \|\widehat F_{0,N}-F_0\|_\infty\longrightarrow0
$
for a continuous CDF $F_0$.  Then
$$
  \liminf_{N\to\infty}\FDP_N
  \ge \frac{\alpha\pi_0F_0(v)}{w}.
$$
\end{lemma}

Figure~\ref{fig:bh-crossing} summarizes the geometry of
Lemmas~\ref{lem:threshold}--\ref{lem:fdp}.  Only a feasible point at $v$ and a
strictly infeasible terminal interval beginning at $w$ are required.

\begin{figure}[htbp]
\centering
\begin{tikzpicture}[x=11.8cm,y=6.3cm,>=Latex]
  \fill[bhgreen!7] (0.31,0) rectangle (0.68,1.0);
  \fill[bhred!6] (0.68,0) rectangle (1.0,1.0);

  \draw[->,thick] (0,0) -- (1.055,0) node[below right] {$t$};
  \draw[->,thick] (0,0) -- (0,1.055) node[above left] {CDF value};
  \draw[very thick,bhgray] (0,0) -- (1,1)
    node[pos=0.82,above,sloped,font=\small] {$t/\alpha$};

  \draw[very thick,bhblue]
    plot[smooth] coordinates {
      (0,0) (0.10,0.18) (0.22,0.34) (0.31,0.46)
      (0.43,0.53) (0.55,0.57) (0.68,0.60)
      (0.82,0.68) (1.00,0.82)
    };
  \node[bhblue,font=\small,fill=white,inner sep=1pt] at (0.86,0.73) {$G_z(t)$};

  \draw[densely dashed,bhgreen] (0.31,0) -- (0.31,0.46);
  \fill[bhgreen] (0.31,0.46) circle (1.5pt);
  \node[above left,bhgreen!70!black,font=\small,align=right]
    at (0.31,0.46) {$G_z(v)>v/\alpha$\\(strictly feasible)};

  \draw[densely dashed,bhred] (0.68,0) -- (0.68,0.68);
  \fill[bhred] (0.68,0.60) circle (1.5pt);
  \node[below right,bhred!75!black,font=\small,align=left]
    at (0.69,0.56) {$G_z(t)<t/\alpha$\\for every $t\in[w,\alpha]$};

  \draw[densely dashed,bhpurple] (0.585,0) -- (0.585,0.585);
  \node[above,bhpurple,font=\small] at (0.585,0.585) {rightmost crossing};

  \foreach \x/\lab in {0/0,0.31/v,0.68/w,1/\alpha}{
    \draw (\x,0.012) -- (\x,-0.012) node[below,font=\small] {$\lab$};
  }
  \node[font=\scriptsize,bhgreen!60!black] at (0.49,0.08)
    {$v\le\liminf\tau_N\le\limsup\tau_N\le w$};
\end{tikzpicture}
\caption{Conceptual BH crossing geometry.  The displayed curve is schematic;
the proof uses only the two strict sign conditions, not uniqueness or
transversality of the crossing.}
\label{fig:bh-crossing}
\end{figure}
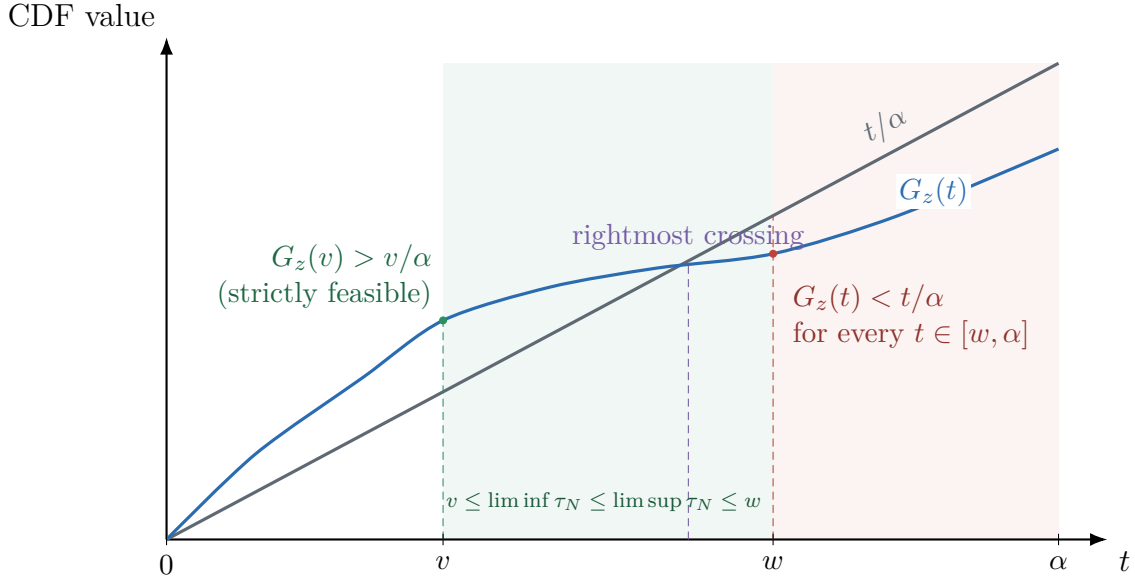

\section{Conditional limiting \texorpdfstring{$p$}{p}-value distributions}

We now recall the factor model introduced earlier. 
Let $a_N\in\R^{100N}$ be the vector whose entries are $3/10$ on the first
block, $-3/10$ on the second block, and $-18/25$ on the third block.  Let
$D_N$ be diagonal, with diagonal entries $91/100$, $91/100$, and $301/625$
on the respective blocks.  Then
$
  \Sigma_N=a_Na_N^{\mathsf T}+D_N.
$
Since every diagonal entry of $D_N$ is strictly positive, $D_N$ and therefore
$\Sigma_N$ are positive definite.  Moreover,
$$
 \left(\frac3{10}\right)^2+\frac{91}{100}=1,
 \qquad
 \left(\frac{18}{25}\right)^2+\frac{301}{625}=1.
$$
Thus every diagonal entry of $\Sigma_N$ is one, so $\Sigma_N$ is a correlation
matrix.  

The mean vector is zero on the first block, $12/5$ on the second
block, and $22/5$ on the third block.  In particular, all nonnull means are
positive.  Every true-null
coordinate is marginally $\mathcal{N}(0,1)$, so its two-sided $p$-value is valid and
uniform on $[0,1]$ marginally.

For $c\ge0$, define
$
  u(c)=2\Phibar(c).
$
Thus $u$ is continuous and strictly decreasing from $1$ to $0$, and
$P_i\le u(c)$ exactly when $|X_i|\ge c$.  For $a\ge0$ and $s>0$, define
\begin{equation}
  Q(c;a,s)
  =\Phibar\left(\frac{c-a}{s}\right)
   +\Phibar\left(\frac{c+a}{s}\right).
  \label{eq:Q-def}
\end{equation}
If $Y\sim \mathcal{N}(m,s^2)$, then
$
  \Pp(|Y|\ge c)=Q(c;|m|,s).
$

Conditional on $Z=z$, the means and standard deviations in the three blocks
are
\begin{align*}
  M_0(z)&=\frac{3z}{10},
  &s_0&=\frac{\sqrt{91}}{10},\\
  M_1(z)&=\frac{12}{5}-\frac{3z}{10},
  &s_1&=\frac{\sqrt{91}}{10},\\
  M_2(z)&=\frac{22}{5}-\frac{18z}{25},
  &s_2&=\frac{\sqrt{301}}{25}.
\end{align*}
Let $F_{g,z}$ be the conditional $p$-value CDF in block $g$.  Equation
\eqref{eq:Q-def} gives
$$
  F_{g,z}(u(c))=Q(c;|M_g(z)|,s_g).
$$
Because the three block proportions are exactly $0.96$, $0.01$, and $0.03$,
the conditional limiting CDF of all $p$-values is
\begin{equation}
  G_z(u(c))
  =0.96\cdot Q(c;|M_0(z)|,s_0)
   +0.01\cdot Q(c;|M_1(z)|,s_1)
   +0.03\cdot Q(c;|M_2(z)|,s_2).
  \label{eq:Gz}
\end{equation}
At $\alpha=0.01$, define
\begin{equation}
  h_z(c)=G_z(u(c))-100u(c).
  \label{eq:hz}
\end{equation}
Thus $h_z(c)\ge0$ is precisely the limiting BH feasibility condition at the
$p$-value threshold $u(c)$.

Conditional on $Z=z$, the $p$-values are independent and identically
distributed within each block.  The Glivenko--Cantelli theorem
\citep{vanDerVaartWellner1996} applied to each of the three blocks therefore
gives
$
  \sup_{0\le t\le1}|\widehat G_N(t)-G_z(t)|\longrightarrow0
$
almost surely under the conditional law given $Z=z$.  Applied to the null
block, it also gives
$$
  \sup_{0\le t\le1}|\widehat F_{0,N}(t)-F_{0,z}(t)|\longrightarrow0.
$$
By the existence of regular conditional laws and Fubini's theorem, these two
convergences hold jointly with unconditional probability one, with $z$ replaced
by the realized value $Z$.  All the limiting CDFs are continuous because the
conditional Gaussian laws have strictly positive variances.

For later use, we record two monotonicity properties of $Q$.  For $c,a\ge0$
and $s>0$,
$$
  \frac{\partial}{\partial c}Q(c;a,s)
  =-\frac1s\left\{
     \phi\left(\frac{c-a}{s}\right)
    +\phi\left(\frac{c+a}{s}\right)
  \right\}<0,
$$
so $Q$ is strictly decreasing in $c$.  Also,
$$
  \frac{\partial}{\partial a}Q(c;a,s)
  =\frac1s\left\{
     \phi\left(\frac{c-a}{s}\right)
    -\phi\left(\frac{c+a}{s}\right)
  \right\}\ge0.
$$
Indeed, $|c-a|\le c+a$, and the standard normal density is decreasing as a
function of the absolute value of its argument.  Hence $Q$ is nondecreasing in
$|m|$.

\section{A finite collection of inequalities suffices}

Let
$
  c_\alpha=\Phi^{-1}(1-\alpha/2).
$
Since BH thresholds never exceed $\alpha$, only $c\ge c_\alpha$ is relevant.
Partition $[-5,5]$ into the one thousand intervals
$$
  B_k=\left[\frac{k}{100},\frac{k+1}{100}\right],
  \qquad -500\le k\le499.
$$
For $g\in\{0,1,2\}$, define the exact extrema
$$
  m_{g,k}^-=\min_{z\in B_k}|M_g(z)|,
  \qquad
  m_{g,k}^+=\max_{z\in B_k}|M_g(z)|.
$$
Because each $M_g$ is affine, these extrema are obtained exactly from the two
endpoints and, when the affine function changes sign on the interval, from the
value zero.

Use the rational $c$-grid
$
  c_j=\frac{j}{1000},
  \, 2575\le j\le10000.
$
We provide below a numerical certificate that 
verifies $u(c_{2575})>0.01$, equivalently
$c_{2575}<c_\alpha$, so this grid starts below the entire relevant
$c$-domain.

For $2575\le j<10000$, define
\begin{align}
 U_{j,k}
 &=0.96\cdot Q(c_j;m_{0,k}^+,s_0)
   +0.01\cdot Q(c_j;m_{1,k}^+,s_1)
   +0.03\cdot Q(c_j;m_{2,k}^+,s_2)
   -100u(c_{j+1}).
 \label{eq:Ujk}
\end{align}
If $z\in B_k$ and $c\in[c_j,c_{j+1}]$, the monotonicities just proved imply
$$
  Q(c;|M_g(z)|,s_g)\le Q(c_j;m_{g,k}^+,s_g),
$$
while the decrease of $u$ gives $-100u(c)\le-100u(c_{j+1})$.  Therefore
\begin{equation}
  h_z(c)\le U_{j,k}
  \quad\text{on }B_k\times[c_j,c_{j+1}].
  \label{eq:upper-rectangle}
\end{equation}

Let $j_k$ be the first grid index for which $U_{j,k}$ is not certified to be
strictly negative, and put
$
  a_k=c_{j_k}.
$
Our numerical certificate verifies $u(a_k)<\alpha$.  In particular, $j_k>2575$.  Every
preceding rectangle has $U_{j,k}<0$, so
\begin{equation}
  h_z(c)<0
  \quad\text{for every }z\in B_k
  \text{ and every }c\in[c_\alpha,a_k].
  \label{eq:negative-prefix}
\end{equation}
The endpoint $a_k$ is included because it is the right endpoint of the last
strictly certified rectangle.

For $j\ge j_k$, define the pointwise lower bound
\begin{align}
 L_{j,k}
 &=0.96\cdot Q(c_j;m_{0,k}^-,s_0)
   +0.01\cdot Q(c_j;m_{1,k}^-,s_1)
   +0.03\cdot Q(c_j;m_{2,k}^-,s_2)
   -100u(c_j).
 \label{eq:Ljk}
\end{align}
Let $\ell_k\ge j_k$ be the first index for which $L_{\ell_k,k}$ is certified
to be strictly positive, and put
$
  b_k=c_{\ell_k}.
$
For every $z\in B_k$, monotonicity in the absolute mean gives
\begin{equation}
  h_z(b_k)\ge L_{\ell_k,k}>0.
  \label{eq:positive-point}
\end{equation}
Because $b_k\ge a_k$, one has $u(b_k)\le u(a_k)<\alpha$.

Fix an outcome in the probability-one event on which both empirical-CDF
convergences hold, and suppose that its realized factor value $z=Z$ lies in
$B_k$.  In Lemma~\ref{lem:threshold}, take
$$
  v=u(b_k),
  \qquad
  w=u(a_k),
  \qquad
  G=G_z.
$$
As $u$ is decreasing, condition \eqref{eq:negative-prefix} is exactly
$G_z(t)<100t=t/\alpha$ for every $t\in[u(a_k),\alpha]$, and
\eqref{eq:positive-point} is exactly
$G_z(u(b_k))>100u(b_k)$.  Lemma~\ref{lem:fdp}, with $\pi_0=0.96$, yields
$$
  \liminf_{N\to\infty}\FDP_N
  \ge
  \frac{0.01\cdot0.96 \cdot F_{0,z}(u(b_k))}{u(a_k)}.
$$
Now
$$
  F_{0,z}(u(b_k))
  =Q(b_k;|M_0(z)|,s_0)
  \ge Q(b_k;m_{0,k}^-,s_0).
$$
Consequently, on this probability-one event, whenever $Z\in B_k$,
\begin{equation}
  \liminf_{N\to\infty}\FDP_N\ge d_k,
  \qquad
  d_k=\frac{0.0096 Q(b_k;m_{0,k}^-,s_0)}{u(a_k)}.
  \label{eq:dk}
\end{equation}

The finite reduction for one $z$-bin is shown in
Figure~\ref{fig:certificate-grid}.  The upper bounds $U_{j,k}$ certify a whole
prefix of infeasible BH thresholds, whereas one lower bound $L_{\ell_k,k}$
provides a feasible point farther out in the Gaussian-tail coordinate.

\begin{figure}[htbp]
\centering
\begin{tikzpicture}[x=8.8cm,y=1.15cm,>=Latex]
  % Main cell and grid.
  \fill[bhred!10] (0,0) rectangle (1,2.5);
  \fill[bhorange!12] (0,2.5) rectangle (1,3.0);
  \draw[thick] (0,0) rectangle (1,5.0);
  \foreach \y in {0.5,1.0,1.5,2.0,2.5,3.0,3.5,4.0,4.5}{
    \draw[bhgray!45] (0,\y) -- (1,\y);
  }
  \draw[very thick,bhred] (0,2.5) -- (1,2.5);
  \draw[very thick,bhgreen] (0,4.15) -- (1,4.15);

  % Axes and labels.
  \draw[->,thick] (-0.02,0) -- (-0.02,5.28) node[above] {$c$};
  \draw[->,thick] (0,-0.05) -- (1.12,-0.05) node[right] {$z$};
  \node[below,font=\small] at (0,-0.05) {$k/100$};
  \node[below,font=\small] at (1,-0.05) {$(k+1)/100$};
  \node[left,font=\small] at (-0.02,0) {$c_\alpha$};
  \node[left,bhred!80!black,font=\small] at (-0.02,2.5) {$a_k$};
  \node[left,bhgreen!70!black,font=\small] at (-0.02,4.15) {$b_k$};

  % Annotations.
  \node[align=center,bhred!80!black,font=\small] at (0.5,1.25)
    {$U_{j,k}<0$ on every\\preceding rectangle};
  \node[align=center,bhorange!85!black,font=\scriptsize] at (0.5,2.75)
    {first rectangle not\\certified negative};
  \fill[bhgreen] (0.55,4.15) circle (2pt);
  \node[right,bhgreen!70!black,font=\small,align=left]
    at (0.57,4.15) {$L_{\ell_k,k}>0$\\for all $z\in B_k$};

  % Decreasing u(c) and threshold bracket.
  \draw[->,thick,bhpurple] (1.18,4.65) -- (1.18,0.35)
    node[midway,right,align=left,font=\small] {$u(c)$ increases\\downward};
  \node[draw=bhpurple,rounded corners=2pt,fill=bhpurple!6,
        align=center,font=\small,text width=7.6cm]
    at (0.5,-1.1)
    {$u(b_k)\le\liminf_N\tau_N\le\limsup_N\tau_N\le u(a_k)$,
     hence $\displaystyle\liminf_N\FDP_N\ge d_k$.};
\end{tikzpicture}
\caption{One certified rectangle column $B_k\times[c_\alpha,10]$.
Monotonicity in $c$ and in the absolute conditional means makes the displayed
finite sign checks valid uniformly over the entire bin.}
\label{fig:certificate-grid}
\end{figure}
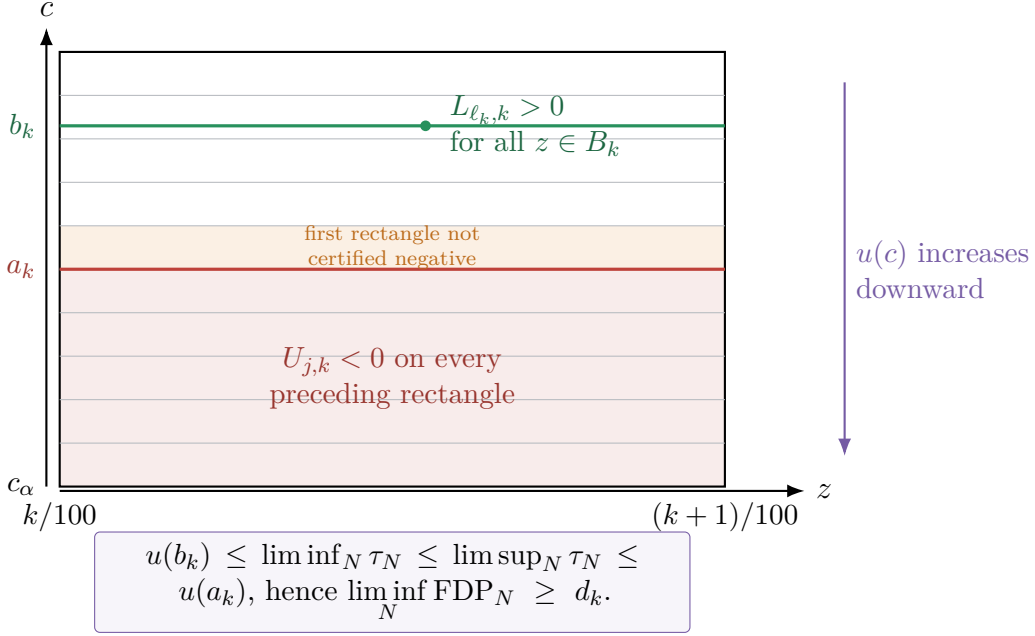

\section{The certified lower bound and completion of the proof}

Every rational input to the certificate---the means, factor loadings, block
weights, $z$-bin endpoints, and $c$-grid points---is specified using ratios of
integers, without binary floating-point literals.  The square roots defining
the residual standard deviations and all Gaussian tails
are evaluated as Arb balls.  Arb performs outward-rounded midpoint-radius
ball arithmetic \citep{Johansson2017}: every computed ball contains the exact
real value.  A strict
comparison is accepted only when the entire resulting ball lies on the stated
side of zero.  Thus the tests $U_{j,k}<0$, $L_{j,k}>0$,
$u(c_{2575})>\alpha$, and $u(a_k)<\alpha$ are rigorous interval statements,
not floating-point heuristics.

The certificate computes the right side of
\begin{equation}
  \sum_{k=-500}^{499}d_k
  \left\{\Phi\left(\frac{k+1}{100}\right)
        -\Phi\left(\frac{k}{100}\right)\right\}.
  \label{eq:cert-sum}
\end{equation}
For readability, the one thousand terms are grouped into ten unit intervals.
The following are certified strict lower bounds; every displayed decimal is
rounded downward:

\begin{center}
\begin{tabular}{@{}cc@{}}
\toprule
Range of $Z$ & Contribution to \eqref{eq:cert-sum} \\
\midrule
$[-5,-4]$ & $0.000006254227057215695292015471184518827$ \\
$[-4,-3]$ & $0.000116801968722543993192086605837856445$ \\
$[-3,-2]$ & $0.000762743726482023098968624717496448543$ \\
$[-2,-1]$ & $0.001839640615452850387738159572547846549$ \\
$[-1,0]$  & $0.002006865275162943558353752660173437476$ \\
$[0,1]$   & $0.001953237173075828884087571780890001262$ \\
$[1,2]$   & $0.001827735787140664069157918588581283062$ \\
$[2,3]$   & $0.001367544916156818165023914540138520994$ \\
$[3,4]$   & $0.000508812722703075293268312075251660515$ \\
$[4,5]$   & $0.000027192658519749972390210373148917831$ \\
\midrule
Total on $[-5,5]$
& $0.010416829070473713117472566385250491510$ \\
\bottomrule
\end{tabular}
\end{center}

Since $0\le\FDP_N\le1$, Fatou's lemma applies.  The bins cover $[-5,5]$ up to
endpoints of Gaussian probability zero.  Equation \eqref{eq:dk} and the
nonnegativity of the contribution from $|Z|>5$ give
\begin{align*}
  \liminf_{N\to\infty}\FDR_N
  =\liminf_{N\to\infty}\E[\FDP_N]
  \ge\E\left[\liminf_{N\to\infty}\FDP_N\right]
  &\ge\sum_{k=-500}^{499}d_k\Pp(Z\in B_k)
  >0.0104.
\end{align*}
This proves Theorem~\ref{thm:main}.

\section{A Monte Carlo experiment}
\label{sec:experiment}

Here we provide a Monte Carlo experiment to support the theoretical analysis.  
However,
a naive Monte Carlo experiment is inefficient:
the conditional FDP is typically modest for central values of the common
factor $Z$, whereas uncommon factor values can produce much larger FDPs and
make a nonnegligible contribution to the expectation.  We therefore stratify
on $Z$ while simulating every residual coordinate and recomputing the BH rule
without approximation.

Let $K=1000$ and define the equiprobable standard-normal strata
$$
  I_k=\left(\Phi^{-1}\left(\frac{k-1}{K}\right),
            \Phi^{-1}\left(\frac{k}{K}\right)\right],
  \qquad 1\le k\le K,
$$
with the usual interpretations at probabilities zero and one.  In
macro-replication $b$, draw independently
$$
  U_{b,k}\sim\operatorname{Unif}\left(\frac{k-1}{K},\frac{k}{K}\right),
  \qquad Z_{b,k}=\Phi^{-1}(U_{b,k}),
$$
one draw from each stratum.  Conditional on each $Z_{b,k}$, generate all
$100N$ Gaussian coordinates from \eqref{eq:block0}--\eqref{eq:block2}, form
the two-sided $p$-values, sort them, apply ordinary BH at $\alpha=0.01$ exactly, and record the resulting $D_{b,k,N}=\FDP_{b,k,N}$.
The macro-replication estimate is
$
  Y_{b,N}=\frac1K\sum_{k=1}^K D_{b,k,N}.
$
Because every stratum has probability $1/K$ and $Z_{b,k}$ has the conditional
law of $Z$ given $Z\in I_k$,
$$
  \E[Y_{b,N}]
  =\frac1K\sum_{k=1}^K\E[\FDP_N\mid Z\in I_k]
  =\E[\FDP_N]
  =\FDR_N.
$$
Thus stratification changes the Monte Carlo variance but not the estimand.
We used $B=100$ independent macro-replications and reported
$
  \widehat{\FDR}_N=\frac1B\sum_{b=1}^B Y_{b,N}$,
$\widehat{\operatorname{MCSE}}
  =\frac{s_Y}{\sqrt B},
$
where $s_Y$ is the sample standard deviation of the $Y_{b,N}$.  The intervals
below are conventional Student-$t$ Monte Carlo intervals with $B-1=99$
degrees of freedom.

The experiment used $100{,}000$ complete Gaussian data sets at each dimension.
% The exact seeds were $20260711$ for $N=50$; $20260712$ and $20260713$ for two
% independent blocks of fifty macro-replications at $N=100$; and $20260720$--
% $20260723$ for four independent blocks of twenty-five macro-replications at
% $N=200$.  
The retained reproducibility bundle uses Python~3.12.3, NumPy~1.26.4,
SciPy~1.14.1, and python-flint~0.8.0.  The complete executable and the
macro-replication outputs are available in the GitHub repository at
\url{https://github.com/dobriban/BH}.

\begin{center}
\small
\begin{tabular}{@{}rrrrrc@{}}
\toprule
$N$ & $m$ & $\widehat{\FDR}_N$ & MCSE & $95\%$ MC interval & $p_+$ \\
\midrule
$50$  & $5{,}000$  & $0.009936$ & $0.000113$
  & $[0.009711,\,0.010161]$ & $0.713$ \\
$100$ & $10{,}000$ & $0.010129$ & $0.000096$
  & $[0.009939,\,0.010320]$ & $0.0905$ \\
$200$ & $20{,}000$ & $0.010359$ & $0.000103$
  & $[0.010155,\,0.010563]$ & $3.56\times10^{-4}$ \\
\bottomrule
\end{tabular}
\end{center}
Here $p_+$ denotes the one-sided Student-$t$ Monte Carlo $p$-value for the
null inequality $\FDR_N\le0.01$.

\begin{figure}[htbp]
\centering
\begin{tikzpicture}
\begin{axis}[
  width=0.84\textwidth,
  height=6.4cm,
  xlabel={$N$},
  ylabel={estimated FDR},
  xmin=35,
  xmax=215,
  ymin=0.00965,
  ymax=0.01062,
  xtick={50,100,200},
  ytick={0.0097,0.0098,0.0099,0.0100,0.0101,0.0102,0.0103,0.0104,0.0105,0.0106},
  scaled y ticks=false,
  yticklabel style={/pgf/number format/fixed,/pgf/number format/precision=4},
  grid=major,
  grid style={draw=bhgray!15},
  legend style={draw=none,fill=none,font=\small,at={(0.02,0.98)},anchor=north west},
]
  \addplot[bhred,dashed,very thick,domain=35:215,samples=2] {0.01};
  \addlegendentry{nominal level $\alpha=0.01$}
  \addplot+[
    only marks,
    mark=*,
    mark size=2.5pt,
    bhblue,
    very thick,
    error bars/.cd,
    y dir=both,
    y explicit
  ] coordinates {
    (50,0.0099361666) +- (0,0.0002248724)
    (100,0.0101292213) +- (0,0.0001902903)
    (200,0.0103589142) +- (0,0.0002037728)
  };
  \addlegendentry{estimate and $95\%$ MC interval}
\end{axis}
\end{tikzpicture}
\caption{Finite-sample stratified Monte Carlo estimates.  The first two
intervals do not resolve the sign of $\FDR_N-\alpha$, while the interval at
$N=200$ lies wholly above the nominal level.}
\label{fig:simulation-results}
\end{figure}
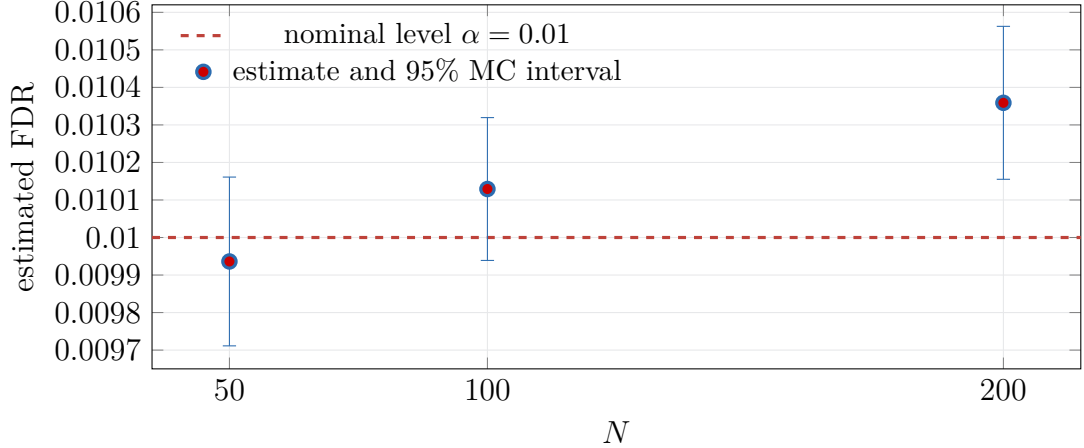

At $N=200$, the excess is
$\widehat{\FDR}_{200}-\alpha=0.0003589$, or about $3.59\%$ of the nominal
level.  The corresponding statistic is $3.495$ Monte Carlo standard errors
above $\alpha$, giving the one-sided $p$-value in the table.  Thus the direct
finite-dimensional experiment provides evidence of the same failure of
control established asymptotically.  At $N=50$ and $N=100$, the intervals
still overlap the nominal level.  This may be because more Monte Carlo
replications are needed, or because the true finite-sample FDR does not exceed
the nominal level at those dimensions.

\section{Discussion}

Several points merit further study.  The example above violates the nominal
level only slightly, and additional numerical searches over related models have
found similarly small violations.  This raises the question of whether a
universal bound exists on the possible inflation of the FDR above its nominal
level.

Moreover, the example uses a large number of tests.  It is therefore important
to determine whether the BH procedure is guaranteed to control the FDR for
smaller numbers of tests and, if not, to obtain bounds that depend explicitly
on the number of tests.  Both questions are directions for future research.

%\clearpage
{\small
\bibliographystyle{plainnat}
\bibliography{ref}

@misc{wang2026adaboost,
      title={AdaBoost Does Not Always Cycle: A Computer-Assisted Counterexample}, 
      author={Erik Y. Wang},
      year={2026},
      eprint={2604.07055},
      archivePrefix={arXiv},
      primaryClass={cs.LG},
      url={https://arxiv.org/abs/2604.07055}, 
}

@misc{abouzaid2026first,
  author       = {Mohammed Abouzaid and Andrew J. Blumberg and Martin Hairer
                  and Joe Kileel and Tamara G. Kolda and Paul D. Nelson
                  and Daniel Spielman and Nikhil Srivastava and Rachel Ward
                  and Shmuel Weinberger and Lauren Williams},
  title        = {First Proof Solutions and Comments},
  year         = {2026},
  month        = feb,
  note         = {Manuscript dated February 14, 2026},
  howpublished = {\url{https://1stproof.org/documents/FirstProofSolutionsComments.pdf}}
}

@misc{openai2026unitdistance,
  author       = {{OpenAI}},
  title        = {Planar Point Sets with Many Unit Distances},
  year         = {2026},
  howpublished = {\url{https://cdn.openai.com/pdf/74c24085-19b0-4534-9c90-465b8e29ad73/unit-distance-proof.pdf}},
  note         = {Accessed 2026-07-13}
}

@misc{openai2026cycle,
  author       = {{OpenAI}},
  title        = {A Proof of the Cycle Double Cover Conjecture},
  year         = {2026},
  howpublished = {\url{https://cdn.openai.com/pdf/04d1d1e4-bc75-476a-97cf-49055cd98d31/cdc_proof.pdf}},
  note         = {Accessed 2026-07-13}
}

@article{dobriban2025solving,
  title={Solving a Research Problem in Mathematical Statistics with AI Assistance},
  author={Dobriban, Edgar},
  journal={arXiv preprint arXiv:2511.18828},
  year={2025}
}

@article{alexeev2025forbidden,
  title={Forbidden Sidon subsets of perfect difference sets, featuring a human-assisted proof},
  author={Alexeev, Boris and Mixon, Dustin G},
  journal={arXiv preprint arXiv:2510.19804},
  year={2025}
}

@article{alexeev2025asymptotically,
  title={Asymptotically optimal approximate Hadamard matrices},
  author={Alexeev, Boris and Jasper, John and Mixon, Dustin G},
  journal={arXiv preprint arXiv:2511.14653},
  year={2025}
}

@article{feldman2025g,
  title={G$\backslash$" odel Test: Can Large Language Models Solve Easy Conjectures?},
  author={Feldman, Moran and Karbasi, Amin},
  journal={arXiv preprint arXiv:2509.18383},
  year={2025}
}

@misc{bubeck2025early,
      title={Early science acceleration experiments with GPT-5}, 
      author={Sébastien Bubeck and Christian Coester and Ronen Eldan and Timothy Gowers and Yin Tat Lee and Alexandru Lupsasca and Mehtaab Sawhney and Robert Scherrer and Mark Sellke and Brian K. Spears and Derya Unutmaz and Kevin Weil and Steven Yin and Nikita Zhivotovskiy},
      year={2025},
      eprint={2511.16072},
      archivePrefix={arXiv},
      primaryClass={cs.CL},
      url={https://arxiv.org/abs/2511.16072}, 
}

@article{salim2025accelerating,
  title={Accelerating mathematical research with language models: A case study of an interaction with GPT-5-Pro on a convex analysis problem},
  author={Salim, Adil},
  journal={arXiv preprint arXiv:2510.26647},
  year={2025}
}

@article{jang2025point,
  title={Point Convergence of Nesterov's Accelerated Gradient Method: An AI-Assisted Proof},
  author={Jang, Uijeong and Ryu, Ernest K},
  journal={arXiv preprint arXiv:2510.23513},
  year={2025}
}

@article{farcomeni2006more,
  title={More powerful control of the false discovery rate under dependence},
  author={Farcomeni, Alessio},
  journal={Statistical Methods and Applications},
  volume={15},
  number={1},
  pages={43--73},
  year={2006},
  publisher={Springer}
}

@article{kim2008effects,
  title={Effects of dependence in high-dimensional multiple testing problems},
  author={Kim, Kyung In and van de Wiel, Mark A},
  journal={BMC bioinformatics},
  volume={9},
  number={1},
  pages={114},
  year={2008},
  publisher={Springer}
}

@article{Benjamini2010,
  author  = {Benjamini, Y.},
  title   = {Discovering the false discovery rate},
  journal = {Journal of the Royal Statistical Society: Series B},
  year    = {2010},
  volume  = {72},
  number  = {4},
  pages   = {405--416},
  doi     = {10.1111/j.1467-9868.2010.00746.x}
}

@article{BenjaminiHochberg1995,
  author  = {Benjamini, Y. and Hochberg, Y.},
  title   = {Controlling the false discovery rate: A practical and powerful
             approach to multiple testing},
  journal = {Journal of the Royal Statistical Society: Series B},
  year    = {1995},
  volume  = {57},
  number  = {1},
  pages   = {289--300},
  doi     = {10.1111/j.2517-6161.1995.tb02031.x}
}

@article{BenjaminiYekutieli2001,
  author  = {Benjamini, Y. and Yekutieli, D.},
  title   = {The control of the false discovery rate in multiple testing under
             dependency},
  journal = {The Annals of Statistics},
  year    = {2001},
  volume  = {29},
  number  = {4},
  pages   = {1165--1188},
  doi     = {10.1214/aos/1013699998}
}

@article{BlanchardRoquain2008,
  author  = {Blanchard, G. and Roquain, E.},
  title   = {Two simple sufficient conditions for {FDR} control},
  journal = {Electronic Journal of Statistics},
  year    = {2008},
  volume  = {2},
  pages   = {963--992},
  doi     = {10.1214/08-EJS180}
}

@article{FinnerDickhausRoters2007,
  author  = {Finner, H. and Dickhaus, T. and Roters, M.},
  title   = {Dependency and false discovery rate: Asymptotics},
  journal = {The Annals of Statistics},
  year    = {2007},
  volume  = {35},
  number  = {4},
  pages   = {1432--1455},
  doi     = {10.1214/009053607000000046}
}

@article{FithianLei2022,
  author  = {Fithian, W. and Lei, L.},
  title   = {Conditional calibration for false discovery rate control under
             dependence},
  journal = {The Annals of Statistics},
  year    = {2022},
  volume  = {50},
  number  = {6},
  pages   = {3091--3118},
  doi     = {10.1214/21-AOS2137}
}

@article{GenoveseWasserman2002,
  author  = {Genovese, C. R. and Wasserman, L.},
  title   = {Operating characteristics and extensions of the false discovery
             rate procedure},
  journal = {Journal of the Royal Statistical Society: Series B},
  year    = {2002},
  volume  = {64},
  number  = {3},
  pages   = {499--517},
  doi     = {10.1111/1467-9868.00347}
}

@article{GenoveseWasserman2004,
  author  = {Genovese, C. R. and Wasserman, L.},
  title   = {A stochastic process approach to false discovery control},
  journal = {The Annals of Statistics},
  year    = {2004},
  volume  = {32},
  number  = {3},
  pages   = {1035--1061},
  doi     = {10.1214/009053604000000283}
}

@misc{GhoshSarkar2025,
  author        = {Ghosh, D. and Sarkar, S. K.},
  title         = {Dependence-aware false discovery rate control in two-sided
                   {Gaussian} mean testing},
  year          = {2025},
  eprint        = {2511.19960},
  archiveprefix = {arXiv},
  note          = {Preprint}
}

@article{Johansson2017,
  author  = {Johansson, F.},
  title   = {{Arb}: Efficient arbitrary-precision midpoint-radius interval
             arithmetic},
  journal = {IEEE Transactions on Computers},
  year    = {2017},
  volume  = {66},
  number  = {8},
  pages   = {1281--1292},
  doi     = {10.1109/TC.2017.2690633}
}

@article{ReinerBenaim2007,
  author  = {Reiner-Benaim, A.},
  title   = {{FDR} control by the {BH} procedure for two-sided correlated tests
             with implications to gene expression data analysis},
  journal = {Biometrical Journal},
  year    = {2007},
  volume  = {49},
  number  = {1},
  pages   = {107--126},
  doi     = {10.1002/bimj.200510313}
}

@article{Sarkar2002,
  author  = {Sarkar, S. K.},
  title   = {Some results on false discovery rate in stepwise multiple testing
             procedures},
  journal = {The Annals of Statistics},
  year    = {2002},
  volume  = {30},
  number  = {1},
  pages   = {239--257},
  doi     = {10.1214/aos/1015362192}
}

@misc{Sarkar2023,
  author        = {Sarkar, S. K.},
  title         = {On controlling the false discovery rate in multiple testing
                   of the means of correlated normals against two-sided
                   alternatives},
  year          = {2023},
  eprint        = {2304.05261},
  archiveprefix = {arXiv},
  note          = {Preprint}
}

@article{SarkarZhang2025,
  author  = {Sarkar, S. K. and Zhang, S.},
  title   = {Shifted {BH} methods for controlling false discovery rate in
             multiple testing of the means of correlated normals against
             two-sided alternatives},
  journal = {Journal of Statistical Planning and Inference},
  year    = {2025},
  volume  = {236},
  pages   = {106238},
  doi     = {10.1016/j.jspi.2024.106238}
}

@article{StoreyTaylorSiegmund2004,
  author  = {Storey, J. D. and Taylor, J. E. and Siegmund, D.},
  title   = {Strong control, conservative point estimation and simultaneous
             conservative consistency of false discovery rates: A unified
             approach},
  journal = {Journal of the Royal Statistical Society: Series B},
  year    = {2004},
  volume  = {66},
  number  = {1},
  pages   = {187--205},
  doi     = {10.1111/j.1467-9868.2004.00439.x}
}

@book{vanDerVaartWellner1996,
  author    = {van der Vaart, A. W. and Wellner, J. A.},
  title     = {Weak Convergence and Empirical Processes},
  publisher = {Springer},
  address   = {New York},
  year      = {1996}
}
}

\appendix
\section{Proofs}

\subsection{Proof of Lemma~\ref{lem:threshold}}

\begin{proof}
Set $H(t)=G(t)-t/\alpha$.  By continuity and the strict inequality
\eqref{eq:no-feasible}, the maximum of $H$ on the compact interval
$[w,\alpha]$ is a strictly negative number.  Hence there is an $\varepsilon>0$
such that
$
  H(t)\le-2\varepsilon
  \quad\text{for all }t\in[w,\alpha].
$
For all sufficiently large $N$, uniform convergence gives
$\|\widehat G_N-G\|_\infty<\varepsilon$.  Therefore
$$
  \widehat G_N(t)-\frac{t}{\alpha}
  \le H(t)+\varepsilon
  \le-\varepsilon<0
$$
for every $t\in[w,\alpha]$.  No BH grid point in that interval is feasible,
which proves $\limsup_N\tau_N\le w$.

For the lower bound, 
$G(v)>\frac{v}{\alpha}$
and continuity give a
$\delta>0$ and an open interval $J$ containing $v$ such that
$H(t)\ge2\delta$
for all $t\in J$.
The mesh of $\mathcal T_N$ is $\alpha/M_N\to0$, so one can choose
$s_N\in\mathcal T_N\cap J$ with $s_N\to v$.  For all sufficiently large $N$,
uniform convergence gives
$$
  \widehat G_N(s_N)-\frac{s_N}{\alpha}
  \ge H(s_N)-\delta
  \ge\delta>0.
$$
Thus $s_N$ is feasible, and the maximal feasible grid point obeys
$\tau_N\ge s_N$.  Taking lower limits proves
$\liminf_N\tau_N\ge v$.
\end{proof}

\subsection{Proof of Lemma~\ref{lem:fdp}}

\begin{proof}
Lemma~\ref{lem:threshold} gives $\liminf_N\tau_N\ge v>0$, so eventually
$\tau_N>0$.  By the definition of the BH threshold,
$
  R_N=\frac{M_N\tau_N}{\alpha}.
$
The number of false rejections is
$
  V_N=\pi_0M_N\widehat F_{0,N}(\tau_N).
$
Consequently,
\begin{equation}
  \FDP_N
  =\frac{V_N}{R_N}
  =\frac{\alpha\pi_0\widehat F_{0,N}(\tau_N)}{\tau_N}.
  \label{eq:fdp-threshold-identity}
\end{equation}
For every $\varepsilon\in(0,v)$, the inequality
$\tau_N\ge v-\varepsilon$ holds eventually.  Monotonicity of empirical CDFs and
uniform convergence then give
$
  \liminf_{N\to\infty}\widehat F_{0,N}(\tau_N)
  \ge F_0(v-\varepsilon).
$
Letting $\varepsilon\downarrow0$ and using continuity of $F_0$ yields
$
  \liminf_{N\to\infty}\widehat F_{0,N}(\tau_N)\ge F_0(v).
$
For any $\varepsilon>0$, these two bounds imply, eventually,
$\widehat F_{0,N}(\tau_N)\ge F_0(v)-\varepsilon$ and
$\tau_N\le w+\varepsilon$.  Substitution in
\eqref{eq:fdp-threshold-identity} gives
$$
  \FDP_N\ge
  \frac{\alpha\pi_0\{F_0(v)-\varepsilon\}}{w+\varepsilon}.
$$
Letting $\varepsilon\downarrow0$ proves the result.
\end{proof}

\section{Complete outward-rounded certificate}\label{app:certificate}

The following Python program is the complete numerical certificate used above.
It requires \texttt{python-flint 0.8.0}.  The specialization to the exact grids used
in the proof deliberately avoids floating-point conversion in all mathematical
inputs and all grid-index calculations.

\begin{lstlisting}[style=certificate]
#!/usr/bin/env python3
"""Outward-rounded certificate for the two-sided Gaussian BH counterexample.

Dependency: python-flint

All model parameters and all subdivision endpoints are exact rationals.
Every transcendental evaluation is an Arb ball with outward rounding.
"""

from flint import arb, ctx

ctx.dps = 40

ALPHA = arb(1) / 100
PI0 = arb(24) / 25
W1 = arb(1) / 100
W2 = arb(3) / 100
R0 = arb(3) / 10
R1 = -arb(3) / 10
R2 = -arb(18) / 25
MU1 = arb(12) / 5
MU2 = arb(22) / 5
S0 = (1 - R0 * R0).sqrt()
S1 = (1 - R1 * R1).sqrt()
S2 = (1 - R2 * R2).sqrt()
SQRT2 = arb(2).sqrt()

Z_DEN = 100
C_DEN = 1000
K_MIN = -500
K_MAX = 500
J_START = 2575
J_STOP = 10000


def normal_upper_tail(x: arb) -> arb:
    return (x / SQRT2).erfc() / 2


def normal_cdf(x: arb) -> arb:
    return 1 - normal_upper_tail(x)


def two_sided_tail(c: arb, abs_mean: arb, sd: arb) -> arb:
    return (
        normal_upper_tail((c - abs_mean) / sd)
        + normal_upper_tail((c + abs_mean) / sd)
    )


def p_threshold(c: arb) -> arb:
    return 2 * normal_upper_tail(c)


def abs_range_of_affine(
    mu: arb, loading: arb, lo: arb, hi: arb
) -> tuple[arb, arb]:
    left = mu + loading * lo
    right = mu + loading * hi
    abs_left = abs(left)
    abs_right = abs(right)
    if (left <= 0 and right >= 0) or (right <= 0 and left >= 0):
        minimum = arb(0)
    else:
        minimum = abs_left if abs_left < abs_right else abs_right
    maximum = abs_left if abs_left > abs_right else abs_right
    return minimum, maximum


def certify_bin(k: int) -> arb:
    z_lo = arb(k) / Z_DEN
    z_hi = arb(k + 1) / Z_DEN

    m0_lo, m0_hi = abs_range_of_affine(arb(0), R0, z_lo, z_hi)
    m1_lo, m1_hi = abs_range_of_affine(MU1, R1, z_lo, z_hi)
    m2_lo, m2_hi = abs_range_of_affine(MU2, R2, z_lo, z_hi)

    c_start = arb(J_START) / C_DEN
    assert p_threshold(c_start) > ALPHA

    j_lower = None
    for j in range(J_START, J_STOP):
        c_j = arb(j) / C_DEN
        c_next = arb(j + 1) / C_DEN
        h_upper = (
            PI0 * two_sided_tail(c_j, m0_hi, S0)
            + W1 * two_sided_tail(c_j, m1_hi, S1)
            + W2 * two_sided_tail(c_j, m2_hi, S2)
            - 100 * p_threshold(c_next)
        )
        if not (h_upper < 0):
            j_lower = j
            break
    if j_lower is None:
        raise RuntimeError(f"No lower bracket in z-bin {k}")

    c_lower = arb(j_lower) / C_DEN
    # This is equivalent to c_lower > c_alpha and ensures that at least
    # one preceding cell was rigorously certified negative.
    assert p_threshold(c_lower) < ALPHA

    j_upper = None
    for j in range(j_lower, J_STOP + 1):
        c_j = arb(j) / C_DEN
        h_lower = (
            PI0 * two_sided_tail(c_j, m0_lo, S0)
            + W1 * two_sided_tail(c_j, m1_lo, S1)
            + W2 * two_sided_tail(c_j, m2_lo, S2)
            - 100 * p_threshold(c_j)
        )
        if h_lower > 0:
            j_upper = j
            break
    if j_upper is None:
        raise RuntimeError(f"No feasible point in z-bin {k}")

    c_upper = arb(j_upper) / C_DEN
    fdp_lower = (
        (PI0 / 100)
        * two_sided_tail(c_upper, m0_lo, S0)
        / p_threshold(c_lower)
    )
    gaussian_mass = normal_cdf(z_hi) - normal_cdf(z_lo)
    return fdp_lower * gaussian_mass


def main() -> None:
    total = arb(0)
    unit_totals: dict[int, arb] = {}

    for k in range(K_MIN, K_MAX):
        contribution = certify_bin(k)
        total += contribution
        unit = k // Z_DEN
        unit_totals[unit] = unit_totals.get(unit, arb(0)) + contribution

    for unit in sorted(unit_totals):
        print(f"z in [{unit},{unit + 1}]: {unit_totals[unit]}")
    print(f"certified total over [-5,5]: {total}")

    assert total > arb("0.0104168290704737131174725663852504915")
    assert total > ALPHA
    print(
        "CERTIFIED: liminf FDR > "
        "0.0104168290704737131174725663852504915 > alpha = 0.01"
    )


if __name__ == "__main__":
    main()
\end{lstlisting}

\end{document}